\documentclass[reqno, 11pt]{amsart}

\usepackage[parfill]{parskip}    
\usepackage{amsrefs, amsmath}
\usepackage{amsfonts, mathrsfs}
\usepackage{amssymb}
\usepackage{amscd, mathtools}
\usepackage{hyperref}
\usepackage{cleveref, multirow} 
\usepackage{fullpage, color}
\usepackage[labelformat=empty]{subfig}
\usepackage{booktabs}
\usepackage{xypic}
\usepackage{enumerate, longtable}
\usepackage{makecell}
\usepackage{IEEEtrantools}
\usepackage{float}

\usepackage{pgf,tikz}\usetikzlibrary{arrows.meta}
\usetikzlibrary{decorations.pathmorphing}
\tikzset{snake it/.style={decorate, decoration=snake}}

\usepackage{multicol}

\numberwithin{equation}{section}

\BibSpec{article}{%
+{}{\sc \PrintAuthors} {author}
+{,}{ \textrm} {title}
+{,}{ \textit} {journal}
+{,}{ } {volume}
+{}{ \IfEmptyBibField{journal}{}{\PrintDate{date}}} {transition}
+{,}{ } {pages}
+{,}{ } {status}
+{.}{} {transition}
+{}{ Preprint available at \tt} {eprint}
+{.}{} {transition}
}

\BibSpec{book}{%
+{}{\sc \PrintAuthors} {author}
+{,}{ \textit} {title}
+{.}{ } {series}
+{,}{ } {volume}
+{.}{ } {publisher}
+{,}{ } {place}
+{,}{ } {date}
+{.}{} {transition}
}

\BibSpec{collection.article}{
+{} {\sc \PrintAuthors} {author}
+{,} { \textrm } {title}
+{,} { in \it \PrintConference} {conference}
+{,} { pp.~} {pages}
+{.} {\PrintBook} {book}
+{,} { \PrintDateB} {date}
+{.} {} {transition}
}

\BibSpec{innerbook}{
+{.} { \emph} {title}
+{.} { } {part}
+{:} { \emph} {subtitle}
+{.} { } {series}
+{,} { } {volume}
+{.} { Edited by \PrintNameList} {editor}
+{.} { Translated by \PrintNameList}{translator}
+{.} { \PrintContributions} {contribution}
+{.} { } {publisher}
+{.} { } {organization}
+{,} { } {address}
+{,} { \PrintEdition} {edition}
+{,} { \PrintDateB} {date}
+{.} { } {note}
}

\allowdisplaybreaks[3]

\crefname{equation}{Eq.}{Eqs.}
\crefname{eqnarray}{Eq.}{Eqs.}
\crefname{algo}{Algorithm}{Algorithms}
\crefname{conj}{Conjecture}{Conjectures}
\crefname{lem}{Lemma}{Lemmas}
\crefname{thm}{Theorem}{Theorems}
\crefname{claim}{Claim}{Claims}
\crefname{rmk}{Remark}{Remarks}
\crefname{prop}{Proposition}{Propositions}
\crefname{section}{Section}{Sections}
\crefname{appendix}{Appendix}{Appendices}
\crefname{cor}{Corollary}{Corollaries}
\crefname{figure}{Figure}{Figures}
\crefname{table}{Table}{Tables}
\crefname{example}{Example}{Examples}
\crefname{prob}{Problem}{Problems}
\crefname{assm}{Assumption}{Assumptions}
\crefname{defn}{Definition}{Definitions}

\newcommand{\bbA}{\mathbb{A}}

\newcommand{\bbZ}{\mathbb{Z}}

\newcommand{\bbC}{\mathbb{C}}
\newcommand{\bbP}{\mathbb{P}}
\newcommand{\bbF}{\mathbb{F}}

\def\bary{\begin{array}} 
\def\eary{\end{array}} 
\def\ben{\begin{enumerate}} 
\def\een{\end{enumerate}}
\def\bit{\begin{itemize}} 
\def\eit{\end{itemize}}

\def\beq{\begin{equation}}                     %  
\def\eeq{\end{equation}}                       % 
\def\bea{\begin{eqnarray}}                     %         % 
\def\eea{\end{eqnarray}}
\def\bary{\begin{array}} 
\def\eary{\end{array}} 
\def\ben{\begin{enumerate}} 
\def\een{\end{enumerate}}
\def\bit{\begin{itemize}} 
\def\eit{\end{itemize}}

\def\IP{{\mathbb P}}

\theoremstyle{plain}

\newtheorem{thm}{Theorem}[section]

\newtheorem{prop}[thm]{Proposition}

\newtheorem*{conj*}{Conjecture}

\newtheorem*{cor*}{Corollary}

\theoremstyle{definition}

\newtheorem{rem}[thm]{Remark}
\newtheorem*{rem*}{Remark}

\newtheorem*{rems*}{Remarks}

\newtheorem{example}{Example}[section]

\newcommand{\GITl}[1]{\backslash \!\! \backslash _{\kern-.2em #1 \kern0.1em}}
\newcommand{\GIT}[1]{/\!\!/_{\kern-.2em #1 \kern0.1em}}

\newcommand{\ev}{\operatorname{ev}}

\def\bred{\begin{color}{red}}
\def\ered{\end{color}}

\def\bes{\begin{subequations}}
\def\ees{\end{subequations}}

\usepackage{tikz}
\usetikzlibrary{decorations.pathreplacing,decorations.pathmorphing,angles,quotes}
\newcommand\F{\mathbb F}

\newcommand\PP{\mathbb P}

\newtheorem{theorem}{Theorem}[section]

\newtheorem{lemma-definition}[theorem]{Lemma-Definition}
\newtheorem{proposition}[theorem]{Proposition}

\theoremstyle{definition}

\newtheorem{defn}[theorem]{Definition}

\newtheorem{convention}[theorem]{Convention}

\theoremstyle{remark}

\numberwithin{equation}{section}
\numberwithin{figure}{section}

\newcommand {\shO}  {\mathcal{O}}

\newcommand {\shX}  {\mathcal{X}}

\newcommand {\ol} {\overline}

\DeclareMathOperator {\hhh} {H}

\newcommand{\sstyle}{\scriptstyle}

\def\mydate{\ifcase\month \or January\or February\or March\or
April\or May\or June\or July\or August\or September\or October\or 
November\or December\fi \space\number\day,\space\number\year}

\newcommand{\ptwo}{\IP^2}

\DeclareMathOperator{\mmm}{M}

\newcommand{\wt}{\widetilde}

\DeclareMathOperator{\ftwo}{\mathbb{F}_2}

\begin{document}

\title{
  Stable maps to Looijenga pairs built from the plane
  }

\author{Michel van Garrel}

\address{\tiny University of Birmingham, School of Mathematics, B15 2TT, Birmingham, United Kingdom}
\email{m.vangarrel@bham.ac.uk}

\begin{abstract}
Choosing a normal crossings anticanonical divisor of $\bbP^2$ leads to four log Calabi--Yau surfaces, three of which are Looijenga pairs. 
In this survey article, I describe how to count rational curves in these that intersect each component of the divisor in one point of maximal tangency.
\end{abstract}

\maketitle
\setcounter{tocdepth}{1}
\tableofcontents

\section{Introduction}

A log Calabi--Yau surface $(Y,D)$ is a surface $Y$ with a choice of normal crossings divisor $D$ such that the complement is Calabi--Yau. $(Y,D)$ is a Looijenga pair if in addition, $D$ has a dimension 0 stratum. Starting with $Y=\bbP^2$, $D$ is of degree 3 and there are 4 ways of building a log Calabi--Yau surface:  $D=D^{\rm toric}$, the toric divisor given as a union of 3 lines; $D=D_1+D_2$, the union of a line $D_1$ and a conic $D_2$ not tangent to $D_1$; $D=D_3$ a nodal cubic; and $D=E$ a smooth cubic. All but the last are Looijenga pairs.

\begin{figure}[h]
\includegraphics[scale=0.2]{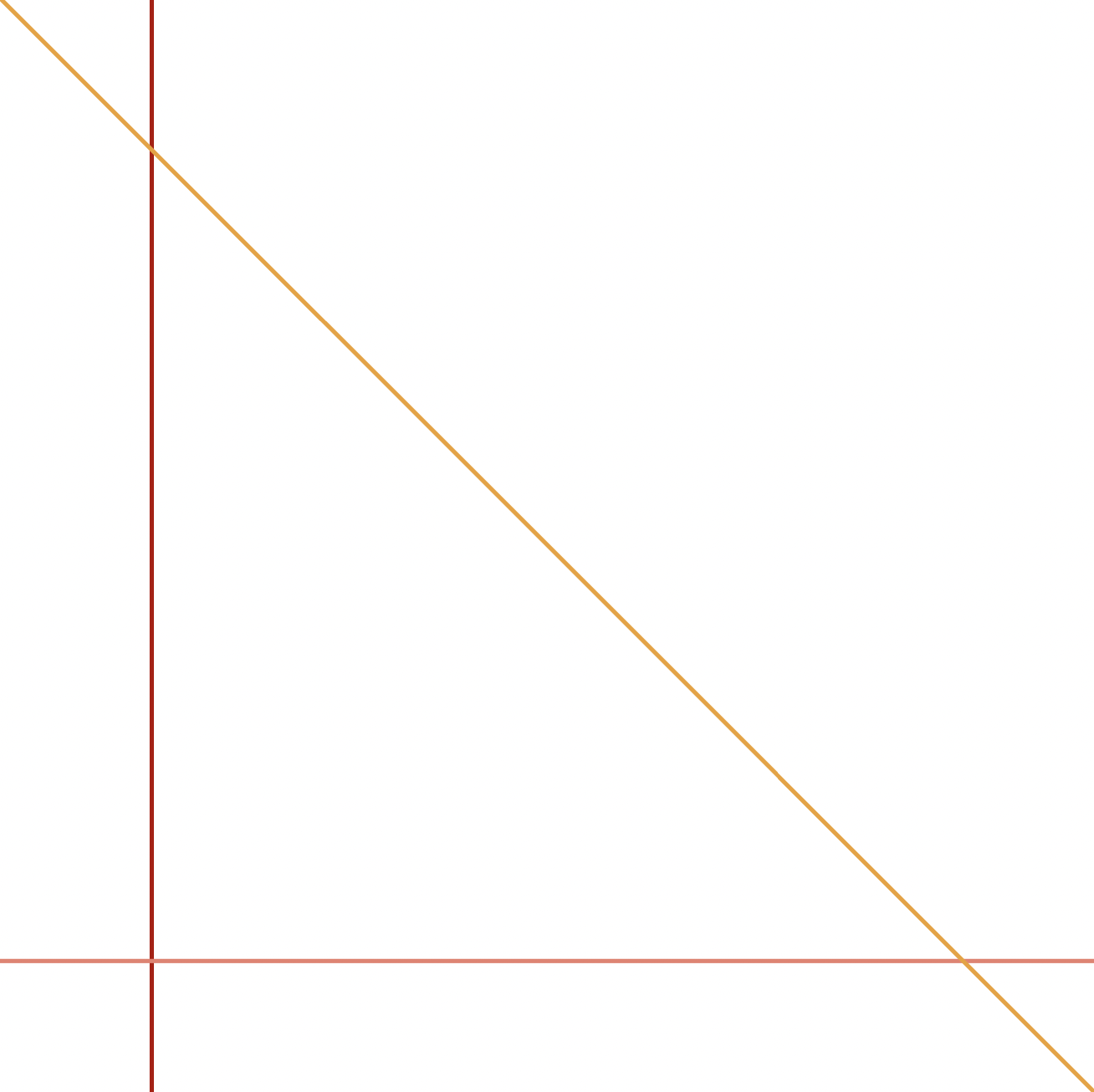}\hspace{20mm}\includegraphics[scale=0.2]{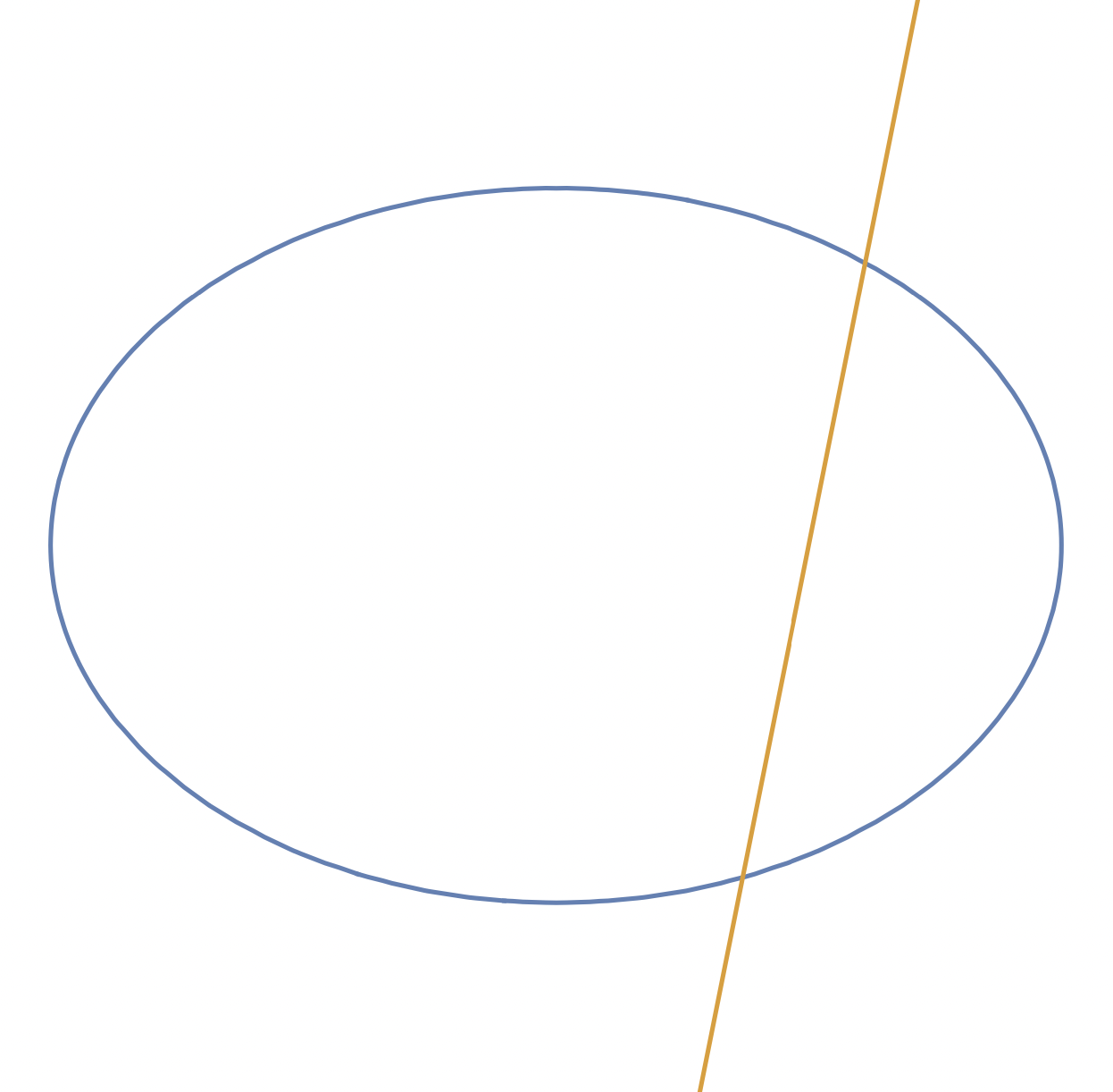}
\caption{The toric divisor $D^{\rm toric}$ and the divisor $D_1+D_2$ formed of a line and a conic.}
\end{figure}

\begin{figure}[h]
\includegraphics[scale=0.2]{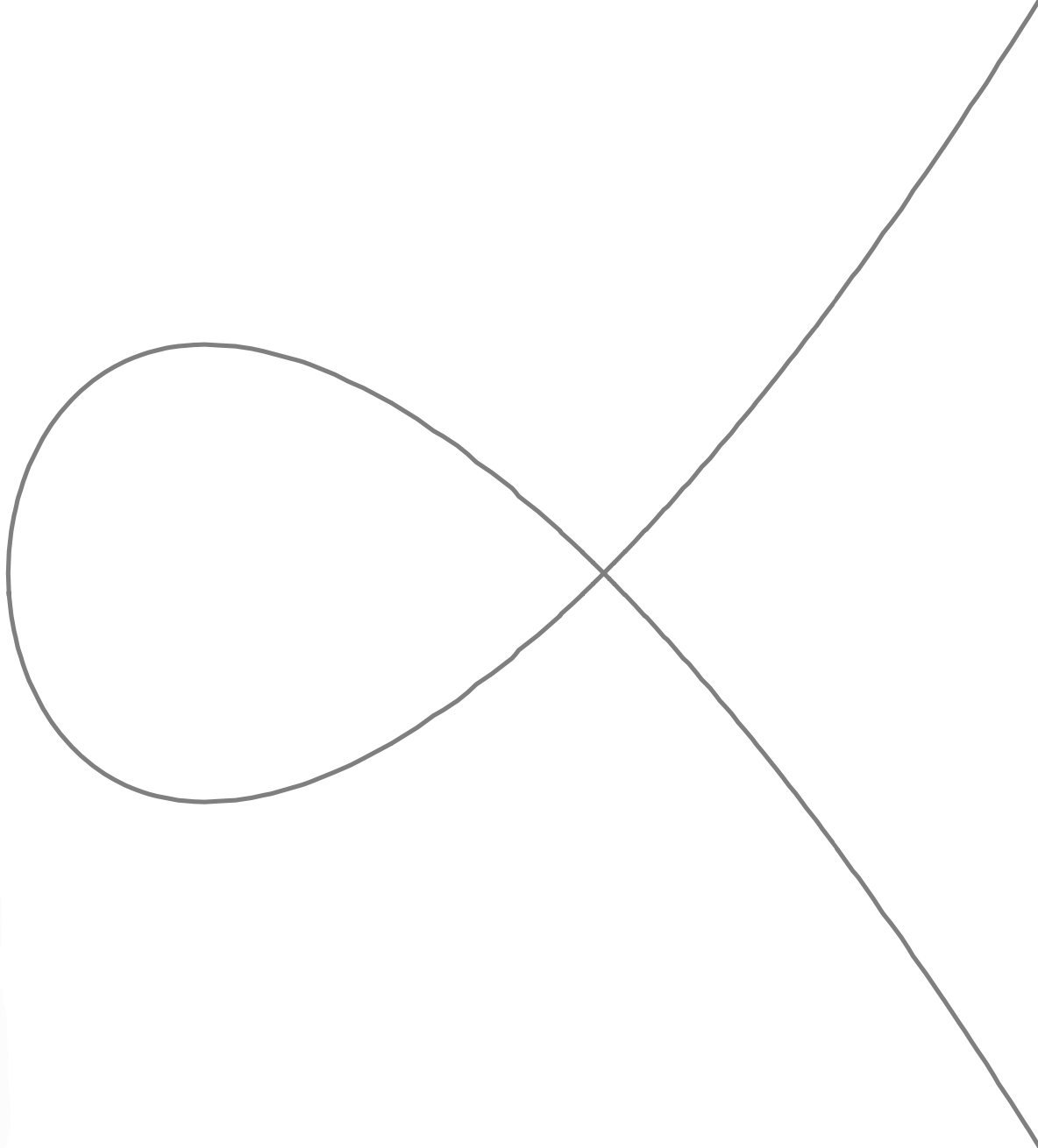}\hspace{20mm}\includegraphics[scale=0.2]{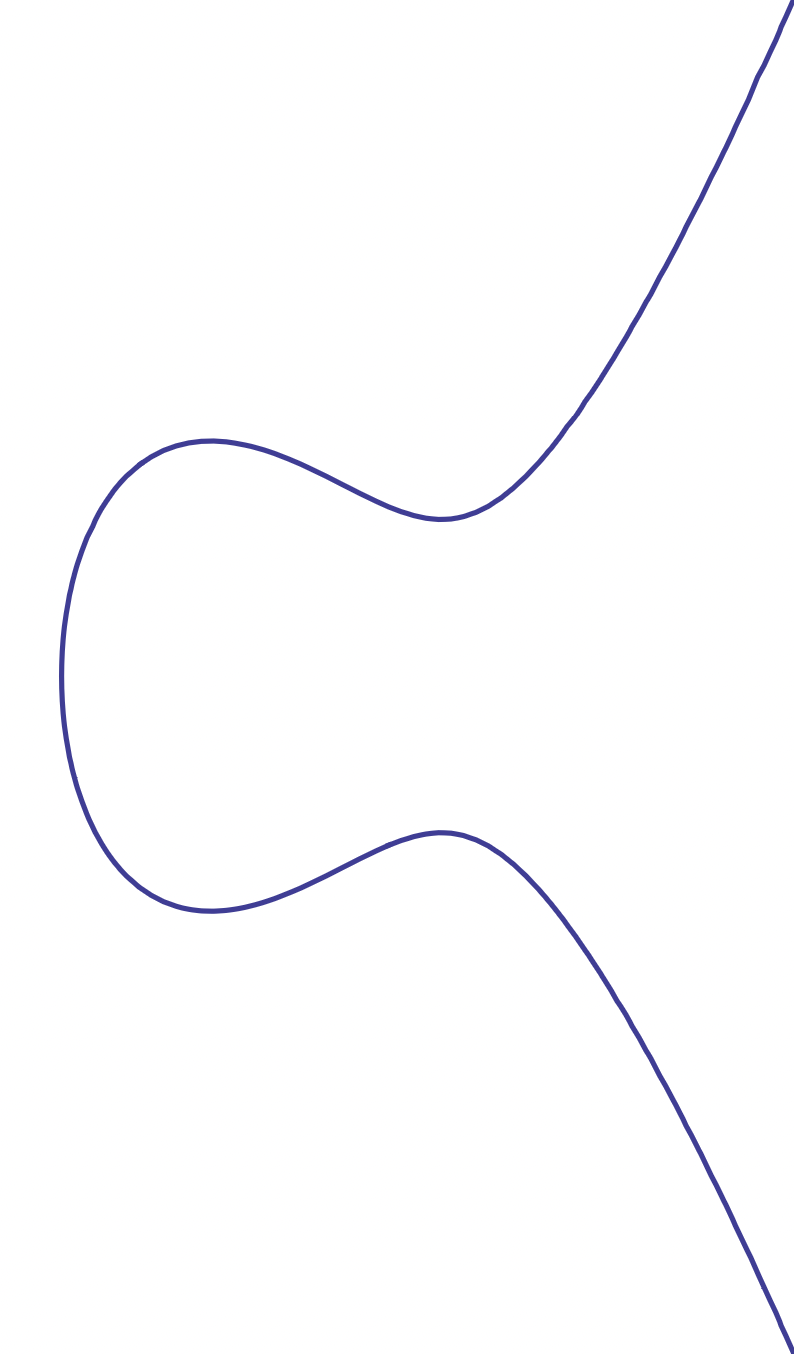}
\caption{The divisors formed by resp.\ a nodal cubic $D_3$ and a smooth cubic $E$.}
\end{figure}

Denote by $l\leq3$ the number of irreducible components of $D$ and let $\bbP^1[l]$ be the complement of $l$ points in $\bbP^1$. Up to automorphism, we have that $\bbP^1[1]=\bbA^1$, $\bbP^1[2]=\bbC^*$ and $\bbP^1[3]=\bbC^*\setminus\{1\}$.
A surprisingly intricate question is to study the geometry of $Y\setminus D$ by counting pseudo-holomorphic maps
\[
\bbP^1[l]\longrightarrow Y \setminus D
\]
of degree $d\in\bbZ_{>0}$ that pass through $l-1$ fixed points. For example, in degree 1 and for $D=E$, there are 9 such $\bbA^1$-curves corresponding to the 9 inflection lines of $E$.

The modern formulation of this count is as the genus 0 maximal tangency log Gromov--Witten invariant $N_d(\bbP^2,D)$ \cite{Chen14,AbramChen14,GS13}. These invariants, more generally the punctured invariants of \cite{ACGSpunct}, play a crucial role in the mirror constructions of intrinsic mirror symmetry \cite{GSintr,GScan}.
In this survey, I describe the methods used to compute the $N_d(\bbP^2,D)$. For the case of $D=D_1+D_2$, I provide a full proof, which may serve as an extensive introduction to the scattering diagram methods of \cite{BBvG2}, see also \cite{B} for an introduction emphasing the links with physics.

For the toric $(\bbP^2,D^{\rm toric})$, we can make use of tropical correspondence results \cite{Mikh05,NiSi,ManRu}, which turn $N_d(\bbP^2,D^{\rm toric})$ into a count of weighted tropical curves.
A special case of \cite[Theorem 3.2]{BBvG1} then states that
\begin{equation}
\label{solution1}
N_d(\bbP^2,D^{\rm toric})=d^2,
\end{equation}
reproven in Proposition \ref{prop1} below.

In the case of a line and a conic, $N_d(\bbP^2,D_1+D_2)$ is obtained by degeneration.
The result of \cite[Example 3.1]{BBvG2}, or Proposition \ref{prop2} below, is that
\begin{equation}
\label{solution2}
N_d(\bbP^2,D_1+D_2)=\binom{2d}{d}.
\end{equation}

The invariants $N_d(\bbP^2,D_3)$ are encoded by the wall-crossing function of the central ray of a local scattering diagram. The terms of this function are conjectured in \cite[Equation (1.3)]{GP10}, which is proven in \cite{Rei11}.
Applying \cite{GPS}, the invariants are given by the equality of power series expansions
\begin{equation}
\label{solution3}
\sum_{d=1}^\infty \, d \, N_d(\bbP^2,D_3) \, x^{d} = 3 \log\left( \sum_{k=0}^\infty \, \frac{1}{3k+1} \, \binom{4k}{k} \, x^{k} \, \right).
\end{equation}

The structure of the invariants of $(\PP^2,E)$ carries similarities with the structure of invariants of K3 surfaces \cite{CGKT2,Bousseau:2020ckw} and thus sometimes $(\PP^2,E)$ is called a log K3 surface.
The computation of the $N_d(\PP^2,E)$ is outside of the scope of this survey. It has a long and rich history. It was studied for the first time in the landmark \cite{MR1844627}, who stated and provided evidence for 3 influential conjectures regarding $(\bbP^2,E)$. The first one is the first appearance of the \emph{log/local correspondence}, which expresses the $N_d(\bbP^2,E)$ as the genus 0 local Gromov--Witten invariants of the local surface given as the total space ${\rm Tot}\left(\shO_{\bbP^2}(-E)\right)$. This conjecture was first proven in \cite{MR1962055}, subsequently generalised in \cite{vGGR}, and then extended in \cite{BBvG1,BBvG2,BBvG3,Bousseau:2020ckw,tseng2020mirror,BNTY,NR22}. The second conjecture is a BPS version of the log/local correspondance for $(\bbP^2,E)$, that is, a correspondence between the integer-valued counts that govern the genus 0 log and local theories. It was proven through the succession of works \cite{CPS,Gra,Bou19a,Bou19b}.

The article \cite{MR1844627} also introduces log mirror mirror symmetry of $(\bbP^2,E)$. The author builds a mirror family and computes the periods on the mirror family. The third conjecture then states that the second order period written in canonical coordinates is a generating function of the $N_d(\bbP^2,E)$. Building on \cite{CPS,Gra,GScan}, a geometric proof of this prediction is given in all dimensions in the upcoming work \cite{vGRS}, see also \cite{vG} for a self-contained proof for the case of $(\bbP^2,E)$.

\subsection{Methods}

The approach I take in this survey is to degenerate to \emph{toric models} for the Looijenga pairs. These are other -- \emph{toric} -- Loiijenga pairs $(\ol{\bbP^2},\ol{D})$ such that $(\bbP^2,D)$ is a logarithmic modification of a blow-up of $(\ol{\bbP^2},\ol{D})$ in smooth points of $\ol{D}$. Concretely this is a diagram
\[ %\label{diag:tm}
\xymatrix{
 & (\wt{\bbP^2},\wt{D}) \ar[ld]_\varphi \ar[rd]^\pi & \\
(\bbP^2,D) & &  (\ol{\bbP^2},\ol{D})
}
\]
such that $\varphi$ is a sequence of blow-ups in singular points of (the total transform of) $D$ and $\pi$ is a toric model, i.e.\ a sequence of blow-ups in smooth points of (the proper transform of) $\ol{D}$.
By \cite[Proposition 1.3]{GHKlog} such toric models always exist. By \cite{AW} $\varphi$ preserves the log Gromov--Witten invariants.

The next step is to perform a sequence of standard degenerations to $(\wt{\bbP^2},\wt{D})$. Applying the degeneration formula \cite{MR1938113,zbMATH07283066,KLR,AF16,Chendeg}, the outcome is that the invariants $N_d(\bbP^2,D)$ are computed in terms of the log Gromov--Witten invariants of $(\ol{\bbP^2},\ol{D})$, which are calculated by tropical correspondence, and the log Gromov--Witten invariants of some standard pieces, which are known. This very general algorithm was formalised in terms of scattering diagram algorithms in \cite{GPS}.

One goal of this note is to unwrap the constructions of \cite{GPS,GHKlog,GHS16} for the examples of $(\bbP^2,D_1+D_2)$ and $(\bbP^2,D_3)$.
For a full treatment, the reader is invited to consult \cite{BBvG2}.

There are other approaches to computing these invariants, all interesting in their own right. Working at the level of moduli spaces these include \cites{vGGR,MR3228298,CGKT1,CGKT2,CGKT3,BN20,NR22} and working with Givental-style mirror symmetry \cites{BN21,tseng2020mirror,BNTY,Y22}.

\addtocontents{toc}{\protect\setcounter{tocdepth}{0}}
\section*{Acknowledgements}
\addtocontents{toc}{\protect\setcounter{tocdepth}{1}}

I thank Dhruv Ranganathan for originally suggesting to study the example of $D_1+D_2$ leading to \cite{BBvG2,BBvG3}. I thank the Nottingham Online Algebraic Geometry Seminar, in particular Al Kasprzyk and Livia Campo, for giving me the opportunity to present these results. I thank my collaborators Pierrick Bousseau, Andrea Brini, Jinwon Choi, Navid Nabijou, Helge Ruddat, Yannik Schüler and Bernd Siebert for various discussions linking to several aspects of this paper. I thank Fenglong You for discussions about parallel methods for computing these invariants and Tim Gräfnitz for discussions on \cite{Gra}.

\section{Genus 0 log Gromov-Witten invariants of maximal tangency}

For each Looijenga pair $(\bbP^2,D)$, endow $\bbP^2$ with the divisorial log structure determined by $D$. The log structure is used as a combinatorial tool to impose tangency conditions along the components of $D$. For a map
\begin{equation}\label{stablelogmap}
f : \bbP^1 \longrightarrow \bbP^2
\end{equation}
this is achieved by requiring that the pullback map on log structures at marked points is given by multiplication by the tangency imposed at that point. Compactifying the space of maps as in \eqref{stablelogmap},
\cite{GS13,Chen14,AbramChen14} construct the moduli space $\ol{\mmm}^{\log}_{m}(\bbP^2,D,d)$ of genus 0 maximally tangent basic stable log maps with $m$ marked points, and its virtual fundamental class
\[
[\ol{\mmm}^{\log}_{m}(\bbP^2,D,d)]^{\rm vir}\in\hhh_{2(l+m-1)}(\ol{\mmm}^{\log}_{m}(\bbP^2,D,d)).
\]
I suppress in the notation the $l$ marked points that carry a tangency and keep track of the $m$ interior points that do not carry a tangency condition. The notion of basicness/minimality selects a universal log structure needed for algebraicity and compactness of the moduli space.

For $j=1,\dots,m$, the moduli space admits evaluation maps at the $j$th marked points
\[
\ev_j : \ol{\mmm}^{\log}_{m}(\bbP^2,D,d) \longrightarrow \bbP^2.
\]
Imposing passing through a fixed point in $\bbP^2$ corresponds to capping the virtual fundamental class with $\ev_j^*([{\rm pt}])$. Doing this $l-1$ times leads to the invariant
\[
N_{d}(\bbP^2,D):=\int_{[\ol{\mmm}^{\log}_{l-1}(\bbP^2,D,d)]^{\rm vir}}  \, \prod_{j=1}^{l-1} \ev_j^*([{\rm pt}]).
\]
Other insertions may be imposed. For example, for the class toric varieties for which each toric divisor is nef, \cite{BBvG1} computes all invariants with point and psi classes.

\begin{example} Fixing a point $P$, there are two lines in the plane that pass through $P$, intersect $D_1$ in one point of order 1 and intersect $D_2$ in one point of order 2, see Figure \ref{lines}.

\begin{figure}[h]
\includegraphics[scale=0.3]{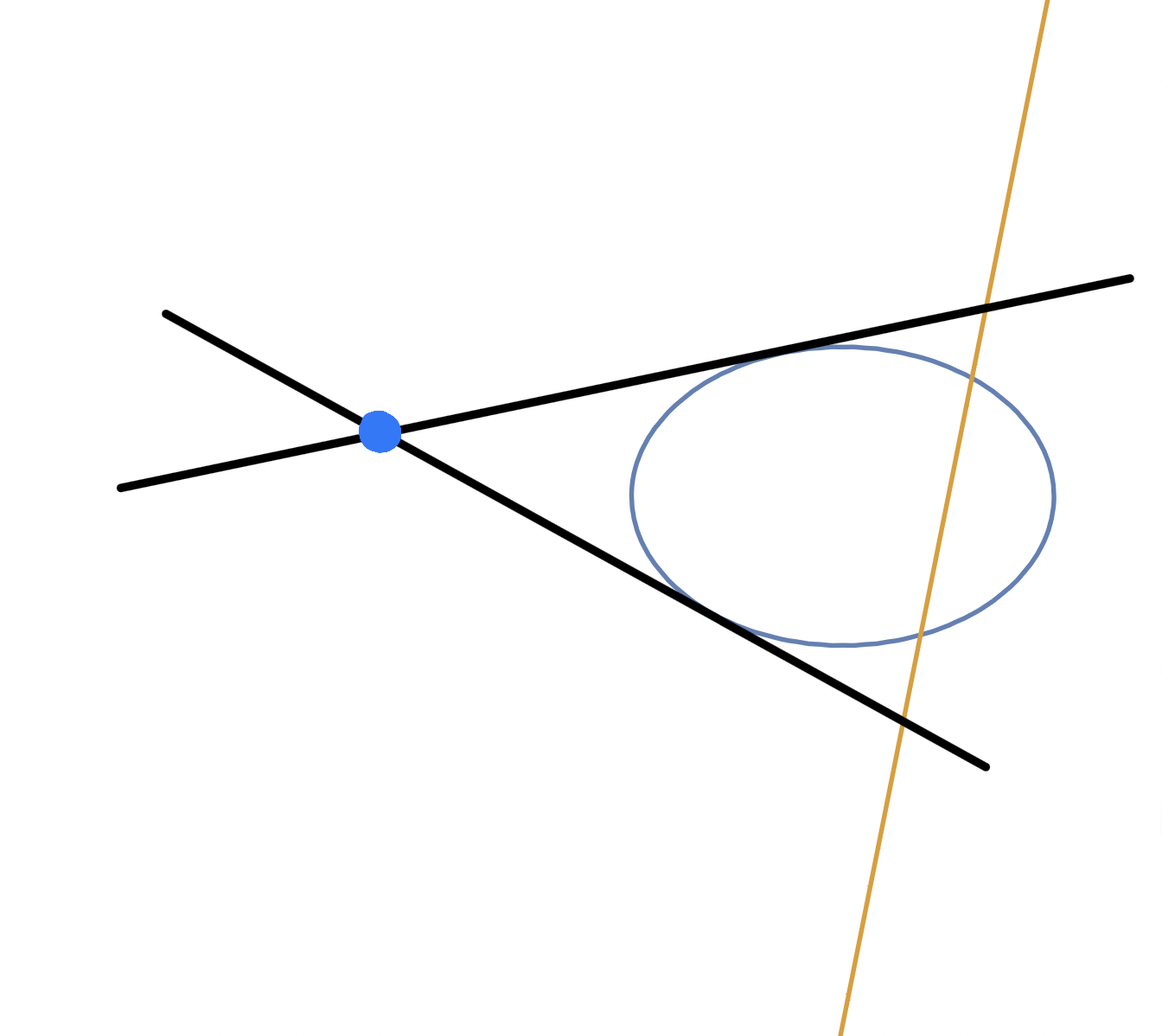}
\caption{The two lines passing through a point and meeting a conic, resp.\ a line, in a point of order 2, resp.\ 1.}
\label{lines}
\end{figure}

\end{example}

\section{Tropical correspondence and the invariants $N_d(\bbP^2,D^{\rm toric})$}

The Looijenga pair $(\bbP^2,D^{\rm toric})$ is toric, thus is its own toric model, and we can compute $N_d(\bbP^2,D^{\rm toric})$ by enumerating weighted tropical curves in the fan $\Xi$ of $(\bbP^2,D^{\rm toric})$.

\begin{defn}
A degree $d$ genus 0 maximal tangency tropical curve $\Gamma$ in $\Xi$ is a graph $\Gamma$ that may have unbounded edges such that:
\begin{enumerate}
\item The vertices $v$ of $\Gamma$ are elements of the support of $\Xi$ separated by edges $e$, which are line segments of rational slope $\in\bbZ^2$ that carry weights $w(e)\geq1$.
\item Balancing condition: For a vertex $v$ of $\Gamma$ with adjacent edge $e$, denote by $u_{(v,e)}$ its primitive outgoing vector. Weighted by $w$, these sum to 0:
\[
\sum_{v\in e} w(e) \, u_{(v,e)}=0.
\]
\item Writing $D^{\rm toric}=T_1+T_2+T_3$ for the toric divisors $T_i$, $\Gamma$ has exactly one unbounded edge of weight $d$ parallel to the ray of $\Xi$ corresponding to $T_i$, and no other unbounded edges.
\item The first Betti number $b_1(\Gamma)=0$.
\end{enumerate}
\end{defn}

For our purposes, a tropical curve can be thought of as keeping track of the tangency behaviour of stable log maps.

\begin{defn}
For a trivalent vertex $v\in \Gamma$ with outgoing edges $e_1,e_2,e_3$, the weight of $v$ is
\[
w(v):=w(e_i) \, w(e_j) \, \big{|} \det \left( u_{(v,e_i)} \, u_{(v,e_j)} \right) \big{|} 
\]
where $i\neq j$ and $\{i,j\}\subset\{1,2,3\}$. This is well-defined by the balancing condition.
The weight of $\Gamma$, $w(\Gamma)$, is given by the products of the weights of each trivalent vertex.

\end{defn}

\begin{proposition}
\label{prop1}
Fixing two general points $P_1,P_2\in\Xi$, there is only one genus 0 degree $d$ maximally tangent tropical curve in $\Xi$ that passes through $P_1$ and $P_2$. It carries weight $d^2$ and hence
\[
N_d(\bbP^2,D^{\rm toric})=d^2.
\]

\end{proposition}

\begin{proof}
The only such tropical curve is given in the following diagram. It has weight $d^2$. The result follows from the tropical correspondence \cite{Mikh05,NiSi,ManRu}.

\begin{center}
\begin{tikzpicture}[smooth, scale=1.2]
\draw[step=1cm,gray,very thin] (-2.5,-2.5) grid (2.5,2.5);
\draw (0,0) to (-2.5,0);
\draw (0,-2.5) to (0,0);
\draw (0,0) to (2.5,2.5);
\draw[thick] (1,-1) to (-2.5,-1);
\draw[thick] (1,-2.5) to (1,-1);
\draw[thick] (1,-1) to (2.5,0.5);
\node at (-0.25,-1.8) {$T_2$};
\node at (1.4,1.7) {$T_3$};
\node at (-1.7,0.2) {$T_1$};
\node at (-1,-1) {$\bullet$};
\node at (1,-2) {$\bullet$};
\node at (-0.75,-0.8) {$P_1$};
\node at (1.25,-1.8) {$P_2$};
\node at (1.2,-2.3) {$d$};
\node at (-2.3,-0.8) {$d$};
\node at (2.4,0.2) {$d$};
\node at (0.9,-0.8) {$d^2$};
\end{tikzpicture}
\end{center}
\end{proof}

\section{Degeneration and the invariants $N_d(\bbP^2,D_1+D_2)$}

\subsection{The toric model}

Denote by $p$ one of the two points of intersection of $D_1$ and $D_2$ and denote by $L$ the line tangent to $D_2$ at $p$. I describe a log modification  that is obtained by successively blowing up points which are dimension 0 strata of the (total transform) of the divisor $D=D_1+D_2$. 
I choose $D_1$ and $L$ to be toric divisors and describe a series of modifications that turn $D_2$ into a toric divisor as well (of a different surface). Every blow-up and blow-down is toric.

\begin{convention}
I denote by the same letter divisors on different surfaces which are related by strict transforms under blow-ups and blow-downs. This process alters the self-intersection numbers.
\end{convention}

Blow up $p$ which leads to the exceptional divisor $F_1$. Then blow up the intersection of $F_1$ with $D_2$ and write $F_2$ for the exceptional divisor. This gives a log Calabi-Yau surface $(\wt{\bbP^2},\wt{D})$, where $\wt{D}$ is the total transform of $D_1+D_2$, see Figure \ref{fig:logmodif}, and where $H$ denotes the pullback of the hyperplane class of $\bbP^2$.

%\vspace{-0.6cm}
\begin{figure}[h]
\begin{tikzpicture}[smooth, scale=1.2]
\draw[step=1cm,gray,very thin] (-2.5,-2.5) grid (2.5,2.5);
\draw[thick] (0,0) to (-2.5,0);
\draw[thick] (0,-2.5) to (0,2.5);
\draw[thick] (0,0) to (2.5,2.5);
\draw[thick] (0,0) to (-2.5,2.5);
%
%\node at (-0.3,-1.7) {$D_2$};
\node at (1.3,1.7) {$D_1$};
\node at (-1.3,1.7) {$F_2$};
\node at (-0.3,1.7) {$F_1$};
\node at (-1.7,0.3) {$L$};
\node at (0.3,-1.7) {$H$};
\end{tikzpicture}
\hspace{0.3cm}
\includegraphics[scale=0.12]{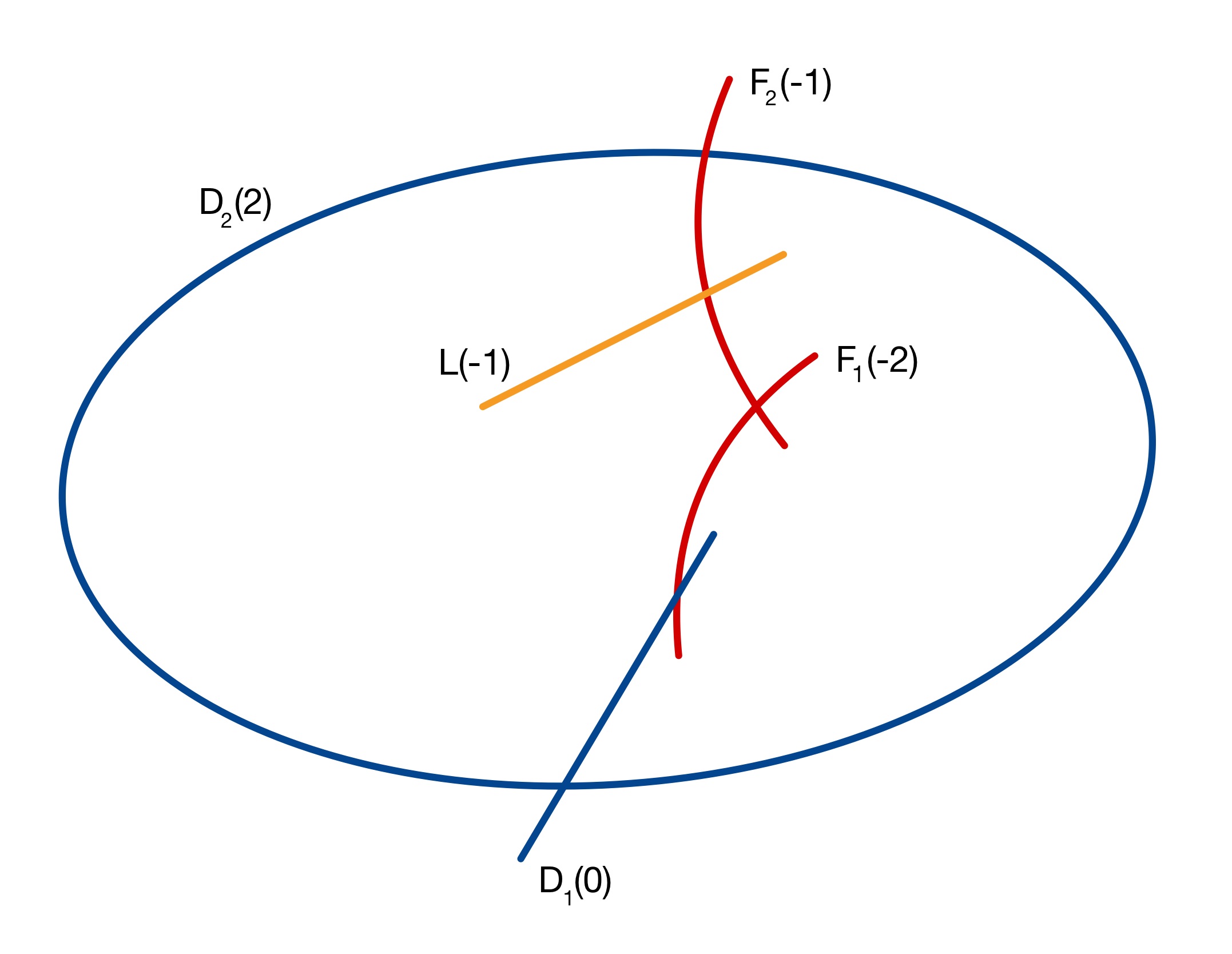}
\caption{The log modification $(\wt{\bbP^2},\wt{D})$.}
\label{fig:logmodif}
\end{figure}

\begin{proposition}
\label{prop:chow1}
The Chow group ${\rm CH}_1(\wt{\bbP^2})$ is generated by $[H]$, $[F_1]$ and $[F_2]$ with relations
\[
[D_1]=[L]+[F_2], \quad [H] = [D_1] + [F_1] + [F_2], \quad [D_2]=[H]+[L],
\]
\[
[H]\cdot [D_1] =  [H] \cdot [L] = 1, \quad [L]^2 = -1, \quad [H]^2 = 1, \quad [D_1]^2=0, \quad [D_2]^2=2, \quad [D_2] \cdot [L] = 0.
\]
%where $[H]$ denotes the pullback under the sequence of blow-ups of the hyperplane classe $[H]\in{\rm CH}_1(\bbP^2)$.
\end{proposition}

\begin{proof}
Immediate either by construction or by the toric relations determined by the fan.
\end{proof}

The toric model $(\ol{\ptwo},\ol{D})$ then is given by blowing down the strict transform of $L$, which is a $(-1)$-curve, resulting in $\ol{\ptwo}=\ftwo$. This leads to the fan on the left-hand side of Figure \ref{fig:toricmodel}. Then $\ol{D}=D_1 + F_1 + F_2 + D_2$ with $F_1$ the $(-2)$-curve, $D_2$ a 2-curve, and $D_1,F_2$ the toric fibres of the bundle projection $\ftwo\to\bbP^1$.
%This gives the fan on the left hand side of Figure \ref{fig:toricmodel}. 
Applying the fan automorphism
\[
\begin{pmatrix}1 & 0 \\ 1 & 1\end{pmatrix} \in {\rm SL}(2,\bbZ)
\]
we obtain the fan on the right hand side of Figure \ref{fig:toricmodel}.
In addition, I label the boundary divisors with their self-intersections and I keep track of $L$ by adding a $\times$ on the ray of $F_2$ indicating that we blow up a smooth point of $F_2$ in order to obtain a log Calabi--Yau surface logarithmically equivalent to $(\bbP^2,D_1+D_2)$.

\textbf{N.B.:} Through this process, $D_2$ became a \emph{toric} divisor.

\begin{figure}[h]
\begin{tikzpicture}[smooth, scale=1.2]
\draw[step=1cm,gray,very thin] (-2.5,-2.5) grid (2.5,2.5);
%\draw[thick] (0,0) to (-2.5,0);
\draw[thick] (0,-2.5) to (0,2.5);
\draw[thick] (0,0) to (2.5,2.5);
\draw[thick] (0,0) to (-2.5,2.5);
\node at (-0.3,-1.7) {$D_2$};
\node at (1.3,1.7) {$D_1$};
\node at (-1.3,1.7) {$F_2$};
\node at (-0.3,1.7) {$F_1$};
%\node at (-1.7,0.3) {$L$};
\node at (-1.7,1.7) {$+$};
\node at (3.7,0) {$\longrightarrow$};
\node at (3.6,0.8) {$\cdot\begin{pmatrix}1 & 0 \\ 1 & 1\end{pmatrix}$};
%\draw[<->,decorate,decoration=snake] (3.5,0) to (4.5,0);
\end{tikzpicture} \hspace{0.8cm}
\begin{tikzpicture}[smooth, scale=1.2]
\draw[step=1cm,gray,very thin] (-2.5,-2.5) grid (2.5,2.5);
\draw[thick] (0,0) to (-2.5,0);
\draw[thick] (0,-2.5) to (0,2.5);
\draw[thick] (0,0) to (1.25,2.5);
\node at (-0.55,-1.7) {$D_2(2)$};
\node at (1.5,1.7) {$D_1(0)$};
\node at (-0.5,1.7) {$F_1(\text{-}2)$};
\node at (-1.5,0.3) {$F_2(0)$};
\node at (-2.2,0) {$\times$};
\end{tikzpicture}
\caption{The toric model of $(\bbP^2,D_1+D_2)$.}
\label{fig:toricmodel}
\end{figure}
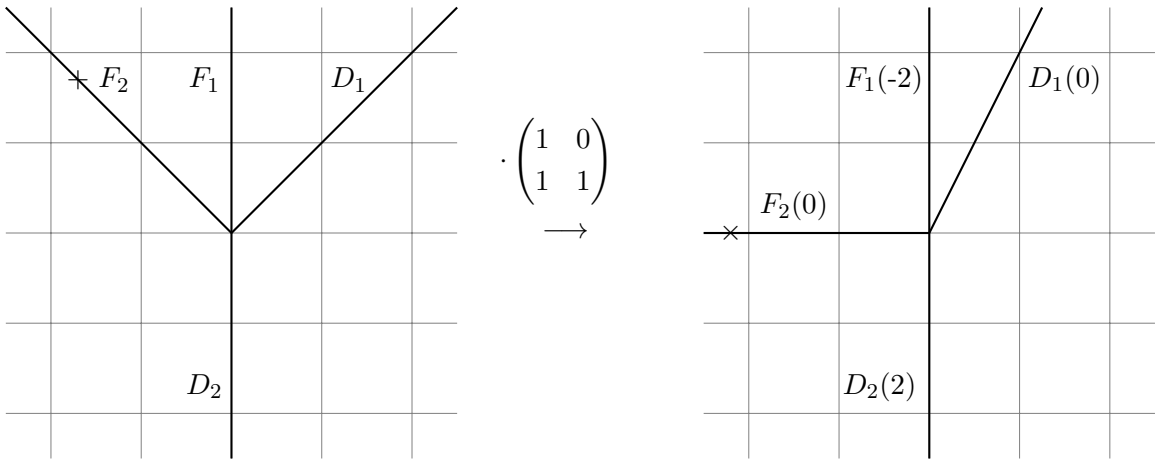

\begin{rem} Note that the fan of a smooth toric surface is uniquely determined by the self-intersections of its toric divisors. Indeed, if $\rho_1,\rho_2,\rho_3$ are primitive generators of successive rays in the fan corresponding to toric divisors $T_1,T_2,T_3$, then $\rho_1+(T_2)^2\rho_2+\rho_3=0$. The collection of these equations determines the fan up to multiplication by an element of ${\rm SL}(2,\bbZ)$.

\end{rem}

\subsection{Calculation by degeneration}
\label{sec:deg}

By invariance of log Gromov-Witten invariants under log modifications,
$N^{\rm log}_d(\bbP^2,D_1+D_2)$ is the genus 0 degree $d[H]$ log Gromov-Witten invariant of $(\wt{\bbP^2},\wt{D})$, passing through one interior point, meeting $D_1$ in one point of tangency $d$ and meeting $D_2$ in one point of tangency $2d$ and with no other conditions. %Such a curve will not intersect $F_1$ or $F_2$ and it will intersect $L$ in $d$ points possibly with multiplicities, none of this is prescribed.

I describe a suitable degeneration of $(\wt{\bbP^2},\wt{D})$ into two components. One component, $(\ol{\bbP^2},\ol{D})$, is toric and thus tropical correspondence results apply. $L$ specialises into the other component, which however will be simple enough so that we can compute its Gromov--Witten invariants.

Start with the trivial family $\bbF_2\times\bbA^1$. The blow up (degeneration to the normal cone)
\[
\shX':={\rm Bl}_{F_2\times\{0\}} \left( \bbF_2\times\bbA^1 \right)
\]
has general fibre $\bbF_2$ and special fibre $\bbF_2\cup_{F_2}\bbF_0$, where $\bbF_0=\bbP^1\times\bbP^1=\bbP(N_{F_2/\bbF_2}\oplus \mathcal{O}_{F_2})$. Denote by $H_1, H_2$ the two effective curve classes generating $\hhh_2(\bbF_0,\bbZ)$. $\bbF_2$ is glued to $\bbF_0$ by identifying $F_2$ with one of the toric divisors of $\bbF_0$ of class $H_1$. Denote by $\mathcal{F}$ the strict transform of $F_2\times\bbA^1$. Choose a section $\mathcal{S}$ of $\mathcal{F}\to\bbA^1$ whose image lies in the smooth locus (interior) of $F_2$. Then
\[
\shX := {\rm Bl}_{\mathcal{S}} \, \shX'
\]
has general fibre $(\wt{\bbP^2},\wt{D})$ and special fibre $\bbF_2\cup_{F_2}Y$, where $Y={\rm Bl}_{pt}(\bbF_0)$ is the blow-up of $\bbF_0$ at a smooth point of the other toric divisor in class $H_1$. If $\mathcal{L}$ is the exceptional divisor of the family $\shX$, by abuse of notation I denote by $L$ the fibre-wise exceptional divisors.

\begin{figure}[h]
\includegraphics[scale=0.17]{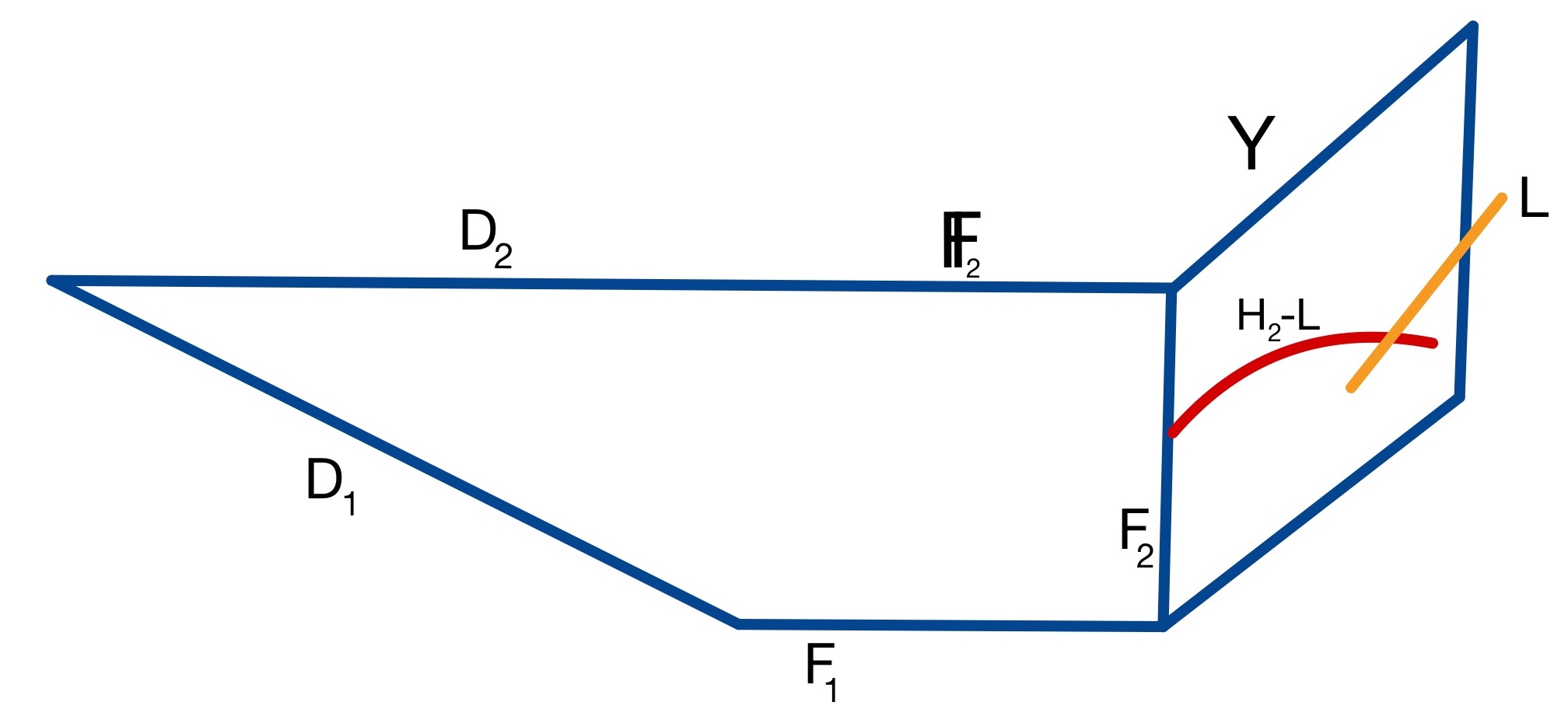}
\caption{The two components of the central fibre.}
\end{figure}

The degeneration formula is a virtual way of keeping track of counts of curves in the degeneration. The first step is to understand the behavior of curves as they specialise from the generic fibre to the special fibre. Start with a curve $C$ in $\wt{\bbP^2}$, or more generally an element in Chow $\alpha\in{\rm CH}_*(\wt{\bbP^2})$. Then \cite{BvBvG1,BvBvG2} construct the prelog Chow ring ${\rm CH}^{\rm prelog}_*(\bbF_2\cup_{F_2}Y)$ along with a \emph{specialisation morphism}
\[
\sigma : {\rm CH}_*(\wt{\bbP^2}) \to {\rm CH}^{\rm prelog}_*(\bbF_2\cup_{F_2}Y).
\]
The ring ${\rm CH}^{\rm prelog}_*(\bbF_2\cup_{F_2}Y)$ is defined as the quotient of tuples $(\alpha_1,\alpha_2)\in{\rm CH}_*(\bbF_2)\oplus{\rm CH}_*(Y)$ that satisfy $\alpha_1|_{F_2}=\alpha_2|_{F_2}$ under the identification of classes on both components which are equal as elements of ${\rm CH}_*(\bbF_2\cup_{F_2}Y)$. 

\begin{proposition}
The prelog Chow ring ${\rm CH}^{\rm prelog}_*(\bbF_2\cup_{F_2}Y)$ is generated by the classes
\begin{itemize}
\item $[(\bbF_2,Y)]$ in degree 2;
\item $[(D_2,H_2)]$, $[(F_2,0)]=[(0,H_1)]$ and $[(0,L)]$ in degree 1;
\item $[({\rm pt},0)]=[(0,{\rm pt})]$ in degree 0,
\end{itemize}
and with intersection products
\[[(D_2,H_2)]\cdot[(F_2,0)]=1, \quad [(D_2,H_2)]\cdot[(0,L)]=[(F_2,0)]\cdot[(0,L)]=0.
\]
\end{proposition}

\begin{proof}
The degrees 2 and 0 are clear.
In the case of curves, ${\rm CH}_1(\bbF_2)$ is generated by $D_2$ and $F_2$. ${\rm CH}_1(Y)$ is generated by $H_1$, $H_2$ and $L$. Moreover, the gluing $\bbF_2\cup_{F_2}Y$ identifies the classes $F_2$ and $H_1$. Then the classes are determined by the tuples that agree on the intersection between both components. Moreover, the intersections can be computed by choosing representatives on each of the components, see \cite[Definition 2.2]{BvBvG1}.
\end{proof}

\begin{proposition}

The specialisation of $[H] \in {\rm CH}_1(\wt{\bbP^2})$ is
\[
\sigma([H])= [(D_2, H_2-L)] \in {\rm CH}_1^{\rm prelog}(\bbF_2\cup_{F_2}Y).
\]

\end{proposition}

\begin{proof}
By Proposition \ref{prop:chow1}, $[H] = [D_1] + [F_1] + [F_2]$.
By construction
\begin{align*}
\sigma([D_1])&=[(F_2,0)],\quad \sigma([F_1])=[(F_1,H_2)]=[(D_2,H_2)]-2[(F_2,0)], \\ \sigma([F_2])&=[(0,[H_1]-[L])]=[(F_2,0)]-[(0,L)].
\end{align*}
Since $\sigma$ is a ring homomorphism, we thus have that
\[
\sigma([H]) = \sigma([D_1]) + \sigma([F_1]) + \sigma([F_2]) = [(F_2+D_2-2F_2+F_2,H_2-L)]=[(D_2,H_2-L)].
\]
\end{proof}

\begin{rem}
One may explicitly check that $\sigma$ is a ring homomorphism, at least with respect to intersections with the class $[H]$. %$\sigma([H])$ respects the intersection products as in Proposition \ref{prop:chow1}.
\end{rem}

%Following the class $d$ through the sequence of blow-ups

%% and blow-down, we get the class $\beta_{\F_2} \in H_2(\F_2,\Z)$ such that
%%\begin{equation}
%%\label{eq:intersections}
%%\beta_{\F_2}\cdot D_1=d, \quad \beta_{\F_2}\cdot F_1=0, \quad \beta_{\F_2}\cdot F_2=d, \quad \beta_{\F_2}\cdot D_2=2d.
%%\end{equation}
%%It follows that $\beta_{\F_2}=d[D_2]$.
%%Remark that the balancing condition
%%\[ d(1,2)+d(-1,0)+2d(0,-1)=0  \]
%%is satisfied.

%%Now, consider the element $d[D_2]\in{\rm CH}_1(\wt{\bbP^2})$. $\sigma$ is a ring homomorphism and thus the intersection products \eqref{eq:intersections} specialise to the central fibre as well. It follows that
%%\[
%%\sigma(d[D_2])=d\big([D_2],[H_2]-[L]\big).
%%\]

Let $d\geq1$. Assume we have two genus 0 curves $C_1$ and $C_2$ in $\bbF_2$ and  $Y$ respectively, of classes $[C_1]=d[D_2]$ and $[C_2]=d([H_2]-[L])$. Then $C_1\cup C_2$ can be smoothed to the generic fibre if and only if $C_1$ and $C_2$ have the same intersection pattern with $F_2$, i.e.\ if and only if $C_1$ and $C_2$ meet $F_2$ at the same points with the same intersection multiplicities. This is called the \emph{pre-deformability} condition.

\begin{prop}
\label{prop2}
\[
N_d(\ptwo,D_1+D_2)=\binom{2d}{d}.
\]
\end{prop}

\begin{proof}

The degeneration formula \cite{MR1938113,zbMATH07283066,KLR,AF16,Chendeg} states that virtually counting genus 0 curves $C$ of class $d[H]\in{\rm CH}_1(\wt{\bbP^2})$ satisfying a list of properties is the same as virtually counting genus 0 curves $C_1$ and $C_2$ of class $([C_1],[C_2])=d[(D_2,H_2-L)]\in{\rm CH}^{\rm prelog}_1(\bbF_2\cup_{F_2}Y)$ that satisfy the pre-deformability condition as well as the specialisation of the list of properties. For us the list of properties is to meet $D_1$, resp. $D_2$, in one point of tangency $d$, resp. $2d$, and to pass through a point in the interior of $\wt{\bbP^2}$. Furthermore, the degeneration formula has an additional multiplicity factor that keeps track of the fact that several different curves in the generic fibre may specialise to the same curve in the special fibre.

As $(-1)$-curves admit no infinitesimal deformations, there is only one curve in $Y$ of class $H_2-L$. It is the strict transform of the fibre in class $H_2$ of $\bbF_0$ that meets the point $\mathcal{S}(0)\in\bbF_0$. As a consequence, the only contributions coming from $Y$ are unions of  multiple covers over this fibre.%, which I denote by $H_2-L$ for convenience.

The precise expression of the degeneration formula (see \cite[Proposition 5.3]{GPS}) is that
\[ N_d(\bbP^2,D_1+D_2) =\sum_{m \vdash d} \prod_{\ell=1}^{d} 
\frac{\ell^{m_\ell}}{m_\ell !} \left( \frac{(-1)^{\ell -1}}{\ell^2} \right)^{m_\ell}
N_{m}(\F_2). \]
\begin{itemize}
\item The sum is over the partitions $m=(m_1,\dots,m_{d})$ of $d$ so that $d=\sum_{\ell=1}^{d} \ell m_\ell$ with $m_\ell\geq0$ and the convention that $0!=1$. This keeps track of the intersection profile of $C_1$, equivalently $C_2$, along $F_2$. In this notation, $C_1$, equivalently $C_2$, has $m_\ell$ points of intersection multiplicity $\ell$ for $0\leq\ell\leq d$ ($m_\ell$ may be zero).
\item Since $H_2-L$ is a $(-1)$-curve, $C_2$ decomposes as a union of multiple covers of degree $\ell$ over $H_2-L$, each meeting $F_2$ in multiplicity $\ell$ at the unique point of intersection of $F_2$ with $H_2-L$. In order to lead to a genus 0 connected curve (more precisely a genus 0 connected stable map), these are glued to a single \emph{connected} curve (more precisely a connected stable map) $C_1$ in $\bbF_2$. For each $m$, there are $m_1+\cdots+m_{d}$ multiple covers over $H_2-L$ that are glued to $C_1$ according to the ramification profile determined by $m$.

\item The factors $\ell^{m_\ell}$ are the multiplicities coming from the degeneration formula. 
\item The $1/m_\ell !$ are the symmetry factors corresponding to the permutations of the $m_\ell$ multiple covers of degree $\ell$ over $H_2-L$. This adjusts for the over-counting that resulted from counting each constellation separately.
\item The $(-1)^{\ell -1}/\ell^2$ are the $Y$-contributions coming from degree $\ell$ covers of the fibre of class $H_2-L$, see \cite[Proposition 5.2]{GPS}. %Since $H_2-L$ is a $(-1)$-curve, these are the only multiple covers of it occur as stable maps.
\item $N_m(\F_2)$ is the genus 0 log Gromov--Witten invariant of $\F_2$ of class $d[D_2]$, counting rational curves that pass through one fixed point in general position in $\F_2$ and have
\begin{itemize}
    \item one point of contact of order $2d$ with $D_2$,
    \item one point of contact of order $d$ with $D_1$,
    \item for every $1\leq \ell\leq d$, $m_\ell$ contact points of contact order $\ell$ with $F_2$, at the same fixed position on $F_2$ corresponding to the intersection of $F_2$ with the unique fibre of $Y$ of class $H_2-L$.
\end{itemize}
\end{itemize}
The invariant $N_m(\F_2)$ is computed tropically. See Figure \ref{fig:tropcurves} for $d=2$ and $m=(2,0)$.

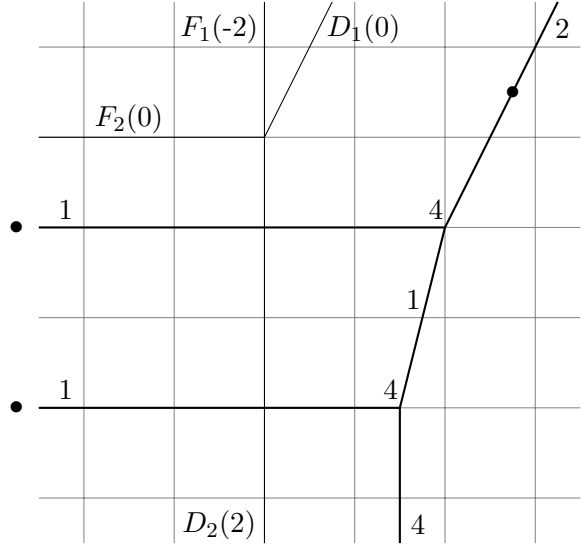
\begin{figure}[h]
\begin{tikzpicture}[smooth, scale=1.2]
\draw[step=1cm,gray,very thin] (-2.5,-3.5) grid (3.5,2.5);
\draw (0,1) to (-2.5,1);
\draw (0,-3.5) to (0,2.5);
\draw (0,1) to (0.75,2.5);
\node at (-0.5,-3.3) {$D_2(2)$};
\node at (1.1,2.2) {$D_1(0)$};
\node at (-0.5,2.2) {$F_1(\text{-}2)$};
\node at (-1.5,1.2) {$F_2(0)$};
%\node at (-2.2,0) {$\times$};
\draw[thick] (1.5,-3.5) to (1.5,-2) to (2,0) to (3.25,2.5);
\draw[thick] (2,0) to (-2.5,0);
\draw[thick] (1.5,-2) to (-2.5,-2);
\node at (1.7,-3.3) {$4$};
\node at (-2.2,-1.8) {$1$};
\node at (-2.2,0.2) {$1$};
\node at (3.3,2.2) {$2$};
\node at (1.65,-0.8) {$1$};
\node at (1.9,0.2) {$4$};
\node at (1.4,-1.8) {$4$};
\node at (2.75,1.5) {$\bullet$};
\node at (-2.75,0) {$\bullet$};
\node at (-2.75,-2) {$\bullet$};
\end{tikzpicture}
\caption{For $d=2$, the only tropical curve with $m=(m_1,m_2)=(2,0)$. The numbers indicate the weights of the edges and vertices. The bullets indicate that we fix a point in the interior of $\bbF_2$ and twice the same point on $F_2$.}
\label{fig:tropcurves}
\end{figure}

There is only one tropical curve contributing to $N_m(\F_2)$. It has $m_1+\cdots+m_{\ell}+\cdots+m_{d}$ infinite edges coming from \emph{fixed} $F_2$-directions, grouped according to $\ell$. The $n$th such curve, for $0< m_1+\cdots+m_{\ell-1}< n \leq m_1+\cdots+m_{\ell}\leq d$ carries weight $\ell$ corresponding to the intersection multiplicity at that point. The tropical curve has one infinite edge for each of $D_1$ and $D_2$, of weights $d$ and $2d$. Requiring the curve to pass through a point in the fan, all other edges and multiplicites are determined by the balancing condition.
Then
\[
N_m(\F_2)=\prod_{\ell= 1}^{d}\left( \frac{1}{\ell} (2d \ell)\right)^{m_\ell}=\prod_{\ell=1}^{d}(2d)^{m_\ell}.
\]
The factors $2d\ell$ are the multiplicities of the vertices. The factors
$1/\ell$ come from the fact that the positions of the contact points with 
$\F_2$ are fixed.
Finally, we obtain
\[ N_d(\bbP^2(1,4)) =\sum_{m \vdash d} \prod_{\ell=1}^{d} 
\frac{1}{m_\ell !} \left( 2d \frac{(-1)^{\ell -1}}{\ell} \right)^{m_\ell} \,.\]
This is the coefficient of $x^d$ in the power series expansion of $\exp (2d \log(1+x))=(1+x)^{2d}$,
hence the result.
\end{proof}

\section{Scattering and the invariants $N_d(\bbP^2,D_3)$}
\label{sec:scat2}

We come to the Looijenga pair formed by the nodal cubic $D_3$. Denote by $P$ the point of self-intersection. First we get a toric model. To do so, blow up $P$ leading to an exceptional divisor $E$, and keep track of the two tangent lines $L_1$ and $L_2$. Blow up the intersections of $E$ with the proper transforms of both $L_1$ and $L_2$ and denote the two new exceptional divisors by $F_1$ and $F_2$.

The proper transforms of $L_1$ and $L_2$ are now $(-1)$-curves, which we may both blow down to obtain the Hirzebruch surface $\bbF_3$. This determines a scattering diagram \cite{GPS}, which is obtained from the fan of the Hirzebruch surface by adding two focus-focus singularities on the two toric fibres.

\begin{figure}[h]
\includegraphics[scale=0.11]{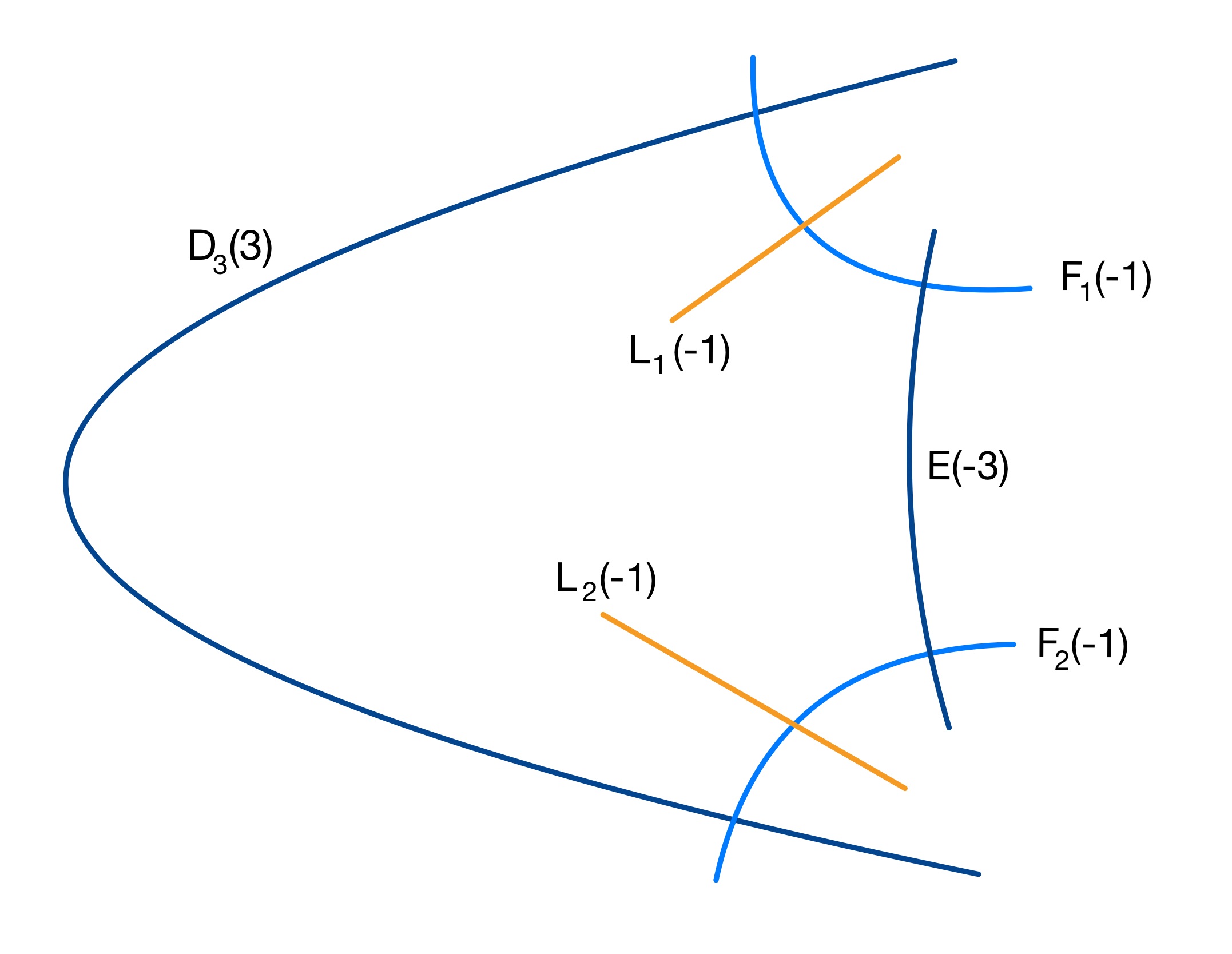}
\hspace{3mm}
\begin{tikzpicture}[smooth, scale=1.2]
\draw[<->] (-2.5,2) to (-1.5,2);
\draw[<->] (-2,1.5) to (-2,2.5);
\node at (-1.67,2.15) {$\scriptstyle{x}$};
\node at (-2.15,2.3) {$\scriptstyle{y}$};
\draw[step=1cm,gray,very thin] (-2.5,-2.5) grid (2.5,3.5);
\draw[thick] (2.5,0) to (-2.5,0);
\draw[thick] (0,-2.5) to (0,3.5);
\draw[thick] (0,0) to (1.16667,3.5);
\node at (-0.55,-1.7) {$D_3(3)$};
\node at (1.5,2.7) {$F_1(0)$};
\node at (-0.5,2.7) {$E(\text{-}3)$};
\node at (-1.5,0.3) {$F_2(0)$};
\node at (-2.3,0) {$\times$};
\node at (0.8,2.4) {$\times$};
\node at (-0.5,0.2) {$\sstyle{1+tx^{-1}}$};
\node[rotate=71.7] at (0.7,1.5) {$\sstyle{1+txy^{3}}$};
\end{tikzpicture}
\caption{The toric model of $(\ptwo,D_3)$.}
\label{fig:ptwonodalcubic}
\end{figure}

The two walls that meet are directed by vectors $v_1$ and $v_2$ that have $ | \det \left( v_1 \, v_2 \right) | =3\neq 1$. As a consequence, there is infinite scattering as described e.g.\ in \cite{GP10}. Each wall carries a wall-crossing function, which identifies the ring of functions on each side of the wall. Scattering is a process of inductively adding rays to the scattering diagram in order to make it consistent order by order. Consistency means that the sequence of automorphisms of a loop around the origin gives the identity up to a given order.

The rays are wall-crossing functions whose coefficients are log Gromov--Witten invariants as shown in \cite{GPS}.
The invariant $N_d(\bbP^2,D_3)$ has one point of contact with $D_3$. As a consequence, by the correspondence results of \cite{GPS} between wall-crossing functions and log Gromov--Witten invariants, $N_d(\bbP^2,D_3)$ is encoded by the wall-crossing function of the central ray of the scattering diagram leading to \eqref{solution3}.

\bibliography{miabiblio}

\end{document}